\documentclass[12pt]{article}
\usepackage{amsmath,amssymb,amsbsy,amsfonts,amsthm,latexsym,amsopn,amstext,
amsxtra,euscript,amscd,graphics,epsfig,color}

\newfont{\teneufm}{eufm10}
\newfont{\seveneufm}{eufm7}
\newfont{\fiveeufm}{eufm5}
\newfam\eufmfam
                   \textfont\eufmfam=\teneufm \scriptfont\eufmfam=\seveneufm
                   \scriptscriptfont\eufmfam=\fiveeufm
\def\bbbc{{\mathchoice {\setbox0=\hbox{$\displaystyle\rm C$}\hbox{\hbox
to0pt{\kern0.4\wd0\vrule height0.9\ht0\hss}\box0}}
{\setbox0=\hbox{$\textstyle\rm C$}\hbox{\hbox
to0pt{\kern0.4\wd0\vrule height0.9\ht0\hss}\box0}}
{\setbox0=\hbox{$\scriptstyle\rm C$}\hbox{\hbox
to0pt{\kern0.4\wd0\vrule height0.9\ht0\hss}\box0}}
{\setbox0=\hbox{$\scriptscriptstyle\rm C$}\hbox{\hbox
to0pt{\kern0.4\wd0\vrule height0.9\ht0\hss}\box0}}}}
\def\bbbq{{\mathchoice {\setbox0=\hbox{$\displaystyle\rm
Q$}\hbox{\raise
0.15\ht0\hbox to0pt{\kern0.4\wd0\vrule height0.8\ht0\hss}\box0}}
{\setbox0=\hbox{$\textstyle\rm Q$}\hbox{\raise
0.15\ht0\hbox to0pt{\kern0.4\wd0\vrule height0.8\ht0\hss}\box0}}
{\setbox0=\hbox{$\scriptstyle\rm Q$}\hbox{\raise
0.15\ht0\hbox to0pt{\kern0.4\wd0\vrule height0.7\ht0\hss}\box0}}
{\setbox0=\hbox{$\scriptscriptstyle\rm Q$}\hbox{\raise
0.15\ht0\hbox to0pt{\kern0.4\wd0\vrule height0.7\ht0\hss}\box0}}}}
\def\bbbt{{\mathchoice {\setbox0=\hbox{$\displaystyle\rm
T$}\hbox{\hbox to0pt{\kern0.3\wd0\vrule height0.9\ht0\hss}\box0}}
{\setbox0=\hbox{$\textstyle\rm T$}\hbox{\hbox
to0pt{\kern0.3\wd0\vrule height0.9\ht0\hss}\box0}}
{\setbox0=\hbox{$\scriptstyle\rm T$}\hbox{\hbox
to0pt{\kern0.3\wd0\vrule height0.9\ht0\hss}\box0}}
{\setbox0=\hbox{$\scriptscriptstyle\rm T$}\hbox{\hbox
to0pt{\kern0.3\wd0\vrule height0.9\ht0\hss}\box0}}}}
\def\bbbs{{\mathchoice
{\setbox0=\hbox{$\displaystyle     \rm S$}\hbox{\raise0.5\ht0\hbox
to0pt{\kern0.35\wd0\vrule height0.45\ht0\hss}\hbox
to0pt{\kern0.55\wd0\vrule height0.5\ht0\hss}\box0}}
{\setbox0=\hbox{$\textstyle        \rm S$}\hbox{\raise0.5\ht0\hbox
to0pt{\kern0.35\wd0\vrule height0.45\ht0\hss}\hbox
to0pt{\kern0.55\wd0\vrule height0.5\ht0\hss}\box0}}
{\setbox0=\hbox{$\scriptstyle      \rm S$}\hbox{\raise0.5\ht0\hbox
to0pt{\kern0.35\wd0\vrule height0.45\ht0\hss}\raise0.05\ht0\hbox
to0pt{\kern0.5\wd0\vrule height0.45\ht0\hss}\box0}}
{\setbox0=\hbox{$\scriptscriptstyle\rm S$}\hbox{\raise0.5\ht0\hbox
to0pt{\kern0.4\wd0\vrule height0.45\ht0\hss}\raise0.05\ht0\hbox
to0pt{\kern0.55\wd0\vrule height0.45\ht0\hss}\box0}}}}
\def\bbbz{{\mathchoice {\hbox{$\sf\textstyle Z\kern-0.4em Z$}}
{\hbox{$\sf\textstyle Z\kern-0.4em Z$}}
{\hbox{$\sf\scriptstyle Z\kern-0.3em Z$}}
{\hbox{$\sf\scriptscriptstyle Z\kern-0.2em Z$}}}}

\newtheorem{theorem}{Theorem}
\newtheorem{lemma}[theorem]{Lemma}

\def\cH{{\mathcal H}}

\def\cM{{\mathcal M}}

\def\cO{{\mathcal O}}

\def\cS{{\mathcal S}}
\def\cT{{\mathcal T}}

\def\({\left(}
\def\){\right)}
\def\[{\left[}
\def\]{\right]}
\def\<{\langle}
\def\>{\rangle}

\def\F{\mathbb{F}}

\def\N{\mathbb{N}}
\def\Z{\mathbb{Z}}
\def\Q{\mathbb{Q}}

\def\mand{\qquad\mbox{and}\qquad}

\graphicspath{{figures/}}

\newcommand{\abs}[1]{\left\vert#1\right\vert}
\newcommand{\set}[1]{\left\{#1\right\}}

\def\pmax{500,000}

\def\End{\mathrm{End}}

\newcommand{\keywords}[1]{{
        \list{}{\advance\topsep by -5ex \relax\small \leftmargin=1cm
        \labelwidth=0pt \listparindent=0pt \itemindent\listparindent
        \rightmargin\leftmargin}\item[\hskip\labelsep \bfseries
        Keywords:] {#1} \endlist}}

\begin{document}

\pagestyle {plain}
\pagenumbering{arabic}


\title{\bf On group structures realized by
elliptic curves over a fixed finite field}


\author{
          \sc{Reza Rezaeian Farashahi} \and \sc{Igor E.~Shparlinski}  \\ \\
          {Department of Computing}\\
          {Macquarie University} \\
          {Sydney, NSW 2109, Australia}\\
          \tt  \{reza,igor\}@ics.mq.edu.au  \\
}

\date{}
\maketitle

\begin{abstract} We obtain explicit formulas for the number of
non-isomorphic elliptic curves with a given group structure
(considered as an abstract abelian group). Moreover, we give
explicit formulas for the number of distinct group structures of all
elliptic curves over a finite field. We use these formulas to derive
some asymptotic estimates and tight upper and lower bounds for
various counting functions related to classification of elliptic
curves accordingly to their group structure. Finally, we present
results of some numerical tests which exhibit several interesting
phenomena in the distribution of group structures. We pose getting
an explanation to these as an open problem.
 \end{abstract}



\section{Introduction}
\label{sec:intro}

\subsection{Background}

Let $\F_q$ be the finite field of characteristic $p$ with $q=p^k$ elements. An elliptic curve $E$ over a finite field $\F_q$ is given by the { \it Weierstrass equation}
\begin{equation}
\label{eq:wei}
y^2+a_1xy+a_3y=x^3+a_2x^2+a_4x+a_6,
\end{equation}
where the coefficients $a_1,a_2,a_3,a_4,a_6$ are in $\F_q$; see~\cite{Silv} for a general background and see~\cite{ACDFLNV} for cryptographic interests on elliptic curves.

As usual, let $E(\F_q)$ be the set of $\F_q$-rational points on elliptic curve $E$ including the point at infinity denoted by $\cO$. It is known, see~\cite{ACDFLNV, Silv, Wash}, that $E(\F_q)$ is a finite {\it Abelian} group with the neutral element $\cO$ and the cardinality of this group satisfies the {\it Hasse-Weil} bound as
$$\abs{\#E(\F_q)-q-1} \le 2\sqrt{q}.$$

It is also known, see~\cite{Silv, Wash}, that the group structure of $E(\F_q)$ is expressed by the group isomorphism
$$E(\F_q) \simeq \Z_m\times \Z_{n},$$
where unique integers $m, n$ satisfy
\begin{equation}
m \mid n \mand m \mid q-1.
\end{equation}

For the prime power $q$ and positive integers $m,n$, let $G(q;m,n)$ be the number of distinct elliptic curves $E$ over $\F_q$  (up to isomorphism over $\F_q$)  such that $E(\F_q) \simeq \Z_m\times \Z_{n}$. Moreover, let $F(q)$ be the the number of distinct group structures of all elliptic curves over the finite field $\F_q$. In this paper, we give explicit formulas for $G(q; m,n)$ and $F(q)$, for all prime powers $q$ and all possible values of $m,n$. We use these
formulas to derive tight upper and lower bounds on  $F(q)$
and aslo an asymptotic formula for the average value
of $F(q)$ on average over prime powers $q \le Q$ as $Q\to \infty$.

Finally, we also present some numerical results concerning
the frequency of the  ``most common'' group structure over $\F_q$,
that is, for
\begin{equation}
\label{eq:Hq}
G(q) = \max_{n,m} G(q; m,n).
\end{equation}
These results reveal several interesting phenomena in the behaviour
of this function and also of the parameters $m$, $n$ and $t = p+1 -
mn$ on which it value is achieved.

Finally, we note the distribution of group structures generated by
elliptic curves generated by all possibke finite field $\F_q$ has been
studied in~\cite{BFS}.

\subsection{Notation}

Throughout the paper, $p$ always denotes a prime and
$q = p^k$ always denotes a prime power.

 Let $t$ be an integer such that $\gcd(t,p)=1$ and $t^2 \le 4q$. Let $\Delta=t^2-4q$ and let $c_t$ be the largest integer such that
$$
c_t^2 \mid \Delta \mand \Delta/c_t^2 \equiv 0 \text{ or } 1 \pmod 4.
$$
Let $s_t$ be the largest integer such that
$$s_t^2 \mid q+1-t \mand s_t \mid q-1.
$$
We note that $s_t \mid c_t$.

For each positive divisor $m$ of $s_t$,  let
$$
\cM_t(m)=\set{e \in \N:   m \mid e,\ e \mid c_t}
$$
and
$$
\cS_t(m)= \cM_t(m) \setminus \bigcup_{ \substack{ l\in \N, \, l> m\\  m \mid l,\, l \mid s_t}} \cM_t(l).
$$

As usual, we use $d(s)$ and $\varphi(s)$
to denote the number of positive integer divisors
and  the Euler  function of $s$, respectively.

Moreover, for every negative integer $D$ with $D \equiv 0$ or $1 \pmod 4$ we denote by $h(D)$ the class number of some quadratic order of discriminant $D$.

The implied constants in the symbols `$O$', `$\ll$'
and `$\gg$' are absolute.
We recall that the notations $U = O(V)$, $U \ll V$ and
$V \gg U$ are all
equivalent to the assertion that the inequality $|U|\le cV$ holds for some
constant $c>0$.

\section{Our Results}

\subsection{Explicit formulas}

For $p>2$, let $\chi_p$ be the quadratic character modulo $p$. So, for a positive integer $x$, we have $\chi_p(x)=0,1$ or $-1$, if $x \equiv 0 \pmod p$, $x \equiv a^2 \pmod p$ for some $a \not\equiv 0 \pmod p$ or $x \not\equiv a^2 \pmod p$ for all $a$, respectively.
Moreover, for $p=2$, let $\chi_p(x)$ equals $0,1$ or $-1$ if $x \equiv 0 \pmod 2$, $x \equiv \pm 1 \pmod 8$ or $x \equiv \pm 3 \pmod 8$, respectively.

\begin{theorem}
\label{thm:Gmn}
Let $q=p^k$ be a power of a prime $p$. Let $m$, $n$ be positive integers. Let $t=q+1-mn$ and $\Delta=t^2-4q$. Then, $G(q;m,n)$, that is, the number of $\F_q$-isomorphism classes of elliptic curves $E$ over $\F_q$ such that $E(\F_q) \simeq \Z_m\times \Z_{n}$, equals:
\begin{enumerate}
  \item  $\displaystyle{\sum_{l\in \cS_{t}(m)} h\(\frac{\Delta}{l^2}\)}$,\ if $\gcd(t,p)=1$, $t^2\le 4q$, $m \mid n$, and $m \mid q-1$
  \item $h(-4p)$,\ if $k$ is odd, $m=1$ and $n=q+1$
  \item $h(-p)$,\ if $k$ is odd, $p \equiv 3 \pmod 4$, $m=2$, and $n=(q+1)/2$
  \item $1$,\ if $k$ is odd, $p=2$ or $3$, $m=1$, and $n=q+1 \pm \sqrt{pq}$
  \item $1-\chi_p(-4)$,\ if $k$ is even, $m=1$, and $n=q+1$
  \item $1-\chi_p(-3)$,\ if $k$ is even, $m=1$, and  $n=q+1 \pm \sqrt{q}$
  \item $(p+6-4\chi_p(-3)-3\chi_p(-4))/12$,\ if $k$ is even, and $m=n=\sqrt{q} \pm 1$
  \item $0$,\ otherwise.
\end{enumerate}
\end{theorem}

The following result gives explicit formulas for the number $F(q)$ of distinct group structures of all elliptic curves over $\F_q$.

\begin{theorem}
\label{thm:G}
Let $q=p^k$ be a power of a prime $p$. For the number $F(q)$ of distinct group structures of all elliptic curves over $\F_q$, we have
\begin{eqnarray*}
F(q)& =& \sum_{ \substack{t \in \Z,\ t^2 \le 4q, \\ \gcd(t,p)=1 } } d(s_t)\\
& & \qquad +~
\left\{
\begin{array}{ll}
\displaystyle{
1+\frac{1-\chi_p(-1)}{2},} &  \text{ if } k \text{ is odd, }
p > 3,\\
\displaystyle{3+\frac{1-\chi_p(-1)}{2},} &  \text{ if } k \text{ is odd, } p = 2,3,\\
\displaystyle{3+\frac{1-\chi_p(-1)}{2}-\chi_p(-3), } &  \text{ if } k \text{ is even, } p >3,\\
5, &  \text{ if } k \text{ is even, } p = 2,3.
\end{array}
\right.
\end{eqnarray*}
\end{theorem}

\subsection{Estimates and average values}

We now present explicit upper and lower bounds on $F(q)$.

\begin{theorem}
\label{thm:LUG}
Let $q=p^k$ be a power of a prime $p$. For the number $F(q)$ of distinct group structures of all elliptic curves over $\F_q$, we have
$$ \frac{2\pi^2}{3}\sqrt{q}\(1-\frac{1}{p}\)+d(q-1)+5 > F(q) >
\left\{
\begin{array}{ll}
2\sqrt{q}+2, &   \text{ if } p=2,\\
\displaystyle 5\sqrt{q}\(1-\frac{1}{p}\)-2, & \text{ if } p \ge 3
\end{array}
\right.
$$
\end{theorem}

We also show that the bounds of Theorem~\ref{thm:LUG} are
asymptotically tight.

\begin{theorem}
\label{thm:limG}
When $q = p^k\to \infty$ via the set of prime powers, we have
\begin{enumerate}
\item \quad $\displaystyle
\limsup_{q \to \infty} \frac{F(q)}{\sqrt{q}\(1-1/p\)}
=  \frac{2\pi^2}{3}$;

\item \quad $\displaystyle
\liminf_{\substack{q \to \infty\\ q~\mathrm{odd}}}
\frac{F(q)}{\sqrt{q}\(1-1/p\)} = 5$;

\item \quad $\displaystyle\liminf_{k \to \infty}
\frac{F(2^k)}{2^{k/2}} = 2$.
\end{enumerate}
\end{theorem}

Finally, we derive an asymptotic formula for the
average value of $F(q)$.

\begin{theorem}
\label{thm:asympG}
For $Q \to \infty$, when $q$ runs over via the set of prime powers, we have
$$
\sum_{q \le Q} F(q) =\(\vartheta+o(1)\)\frac{Q^{3/2}}{\log Q},
$$
where
$$
\vartheta = \frac{8}{3}
\sum_{m =1}^\infty \frac{1}{m^2 \varphi(m)} = 3.682609\ldots.
$$
\end{theorem}

Our argument can also be used to obtain
an explicit bound on the error term in Theorem~\ref{thm:asympG}.

\section{Preliminaries}

\subsection{Endomorphism Rings}

Let $E$ be an elliptic curve over $\F_q$ of characteristic $p$. Let $N=\#E(\F_q)$ and $t=q+1-N$. Let $\pi$ denote the {\it Frobenius} endomorphism on $E$, that is given by
$$\pi: (x,y) \mapsto (x^q, y^q).$$
We note that, $\pi$ is the root of the characteristic polynomial $X^2-tX+q$ in the ring of $\F_q$-endomorphisms of $E$;
This ring is denoted by $\End_{\F_q}(E)$. Moreover, by
$\End(E)=\End_{\overline{\F}_q}(E)$ we denote full endomorphism ring, that is, the ring of $\overline{\F}_q$-endomorphisms of $E$. Let $\Delta=t^2-4q$ be the discriminant of the characteristic polynomial of $E$.

Suppose $\gcd(t,p)=1$. Then, $E$ is called an ordinary elliptic curve. We have $\End(E)=\End_{\F_q}(E)$. Moreover, $\End(E)$ is isomorphic to some order $O$ in the quadratic imaginary field $K=\Q(\sqrt{\Delta})$.
In particular, we have
$$\Z\[\pi\]=\Z\[\frac{\Delta+\sqrt{\Delta}}{2}\] \subseteq \End(E) \subseteq O_K, $$
where $O_K$ is the maximal order in $K$, that is, the ring of algebraic integers of $K$.

 Let $c_t=[O_K : \Z[\pi]]$ be the conductor of $\Z[\pi]$, that is the largest integer such that $\Delta/c_t^2 \equiv 0,1 \pmod 4$. Then,  $\Delta_K=\Delta/c_t^2$, called the fundamental discriminant, is the discriminant of the field $K$. Also, $O_K=\Z\[\frac{\Delta_K+\sqrt{\Delta_K}}{2}\]$. We note that, $O=\Z+cO_K$, where the conductor $c=[O_K : O]$ is a divisor of $c_t$. Furthermore, $\Delta=c^2\Delta_K$ is the discriminant of $O$, so the order $O$ is uniquely determined by its discriminant and denoted by $O(\Delta)$.
  We let $h(O)$ be the class number of $O$ which is also denoted by $h(\Delta)$.

Now, suppose $p \mid t$. Then, $E$ is called a supersingular elliptic curve. Let $\Q_{\infty, p}$ denote the unique quaternion algebra over $\Q$ which is only ramified at $p$ and~$\infty$. Then, $\End_{\F_q}(E)$ is either a quadratic order in $K=\Q(\sqrt{\Delta})$ or a maximal order in $\Q_{\infty,p}$. Moreover, $\End(E)$ is a maximal order in $\Q_{\infty,p}$; see~\cite{Sch, Water} or~\cite{Silv}.

\subsection{Isogeny calsses}
\label{sec:iso}
Two elliptic curves over $\F_q$ are called {\it isogenous} over $\F_q$ if and only if they have the same number of points over $\F_q$.
The number of $\F_q$-rational points of the elliptic curve $E$ over $\F_q$ satisfies the {\it Hasse-Weil} bound. On the other hand, {\it Deuring-Waterhouse} theorem, see~\cite{Water, Wash}, describes all possible values of $N$ that can be the cardinality of $E(\F_q)$, for some elliptic curve $E$ over $\F_q$.

\begin{lemma}
Let $q=p^k$ be a power of a prime $p$. Let $t\in \Z$ and let $N=q+1-t$. The integer $N$ is the cardinality of $E(\F_q)$, for some elliptic curve $E$ over $\F_q$, if and only if one of the following conditions is satisfied:
\begin{enumerate}
  \item $t^2\le 4q$ and $\gcd(t,p)=1$
  \item $k$ is odd and $t=0$
  \item $k$ is odd,  $t=\pm \sqrt{pq}$, $p=2$ or $3$
  \item $k$ is even, $t=0$, $p \not\equiv 1 \pmod 4$
  \item $k$ is even, $t=\pm \sqrt{q}$, $p \not\equiv 1 \pmod 3$
  \item $k$ is even, $t=\pm 2\sqrt{q}$.
\end{enumerate}
\end{lemma}
\begin{proof}
See~\cite{Water, Wash}.
\end{proof}

Let $\Delta$ be a negative integer with $\Delta \equiv 0$ or $1 \pmod 4$ and let $c$ be the largest integer such that $c^2 \mid \Delta$ and $\Delta/c^2 \equiv 0$ or $1 \pmod 4$. Let $H(\Delta)$ denote the {\it Kronecker class number\/} of $\Delta$. We have
$$H(\Delta)=\sum_{l \mid c,\ l>0} h\(\frac{\Delta}{l^2}\).$$

Let $I(q;N)$ be  the number of distinct elliptic curves $E$ over $\F_q$  (up to isomorphism over $\F_q$)  such that $\# E(\F_p) = N$. The following lemma gives  explicit formulas for the values of $I(q;N)$.

\begin{lemma}
\label{lem:isoN}
Let $q=p^k$ be a power of a prime $p$. Let $t\in \Z$ and let $N=q+1-t$.
Then, $I(q;N)$, that is, the number of $\F_q$-isomorphism classes of elliptic curves $E$ over $\F_q$ with $\#E(\F_q)=N$, equals:
\begin{enumerate}
  \item $H(t^2-4q)$,\ if $t^2\le 4q$ and $\gcd(t,p)=1$
  \item $H(-4p)$,\ if $k$ is odd and $t=0$
  \item $1$,\  if $k$ is odd,  $t=\pm \sqrt{pq}$, $p=2$ or $3$
  \item $1-\chi_p(-4)$,\ if $k$ is even, $t=0$, $p \not\equiv 1 \pmod 4$
  \item $1-\chi_p(-3)$,\ if $k$ is even, $t=\pm \sqrt{q}$, $p \not\equiv 1 \pmod 3$
  \item $(p+6-4\chi_p(-3)-3\chi_p(-4))/12$,\ if  $k$ is even, $t=\pm 2\sqrt{q}$
  \item $0$,\ otherwise.
\end{enumerate}
\end{lemma}
\begin{proof}
See~\cite[Theorem 4.6]{Sch}.
\end{proof}

\subsection{Group structures}

The group of $\F_q$-rational points on the elliptic curve $E$ over $\F_q$ is isomorphic to the group $\Z_m\times \Z_{n}$, with unique integers $m, n$ such that $m \mid n$ and $m \mid q-1$. We note that, every group $\Z_m\times \Z_{n}$, with integers $m, n$ satisfy later conditions,
may not occur as the group $E(\F_q)$ for some elliptic curve $E$ over $\F_q$. The following theorem describes the possible group structures for elliptic curves over finite fields, see~\cite{TsVl}.

\begin{lemma}
\label{lem:GS}
Let $q=p^k$ be a power of a prime $p$. Let $m$, $n$ be positive integers with $m \le n$. Let $t=q+1-mn$. There is an elliptic curve $E$ over~$\F_q$ such that
$E(\F_q) \simeq \Z_m \times \Z_n$ if and only if one of the following holds:
\begin{enumerate}
  \item $\gcd(t,p)=1$, $t^2 \le 4q$, $m \mid n$ and $m \mid q-1$
  \item $k$ is odd, $t=0$, $p \not\equiv 3 \pmod 4$, and $m=1$
  \item $k$ is odd, $t=0$, $p \equiv 3 \pmod 4$, and $m=1$ or $2$
  \item $k$ is odd, $t=\pm \sqrt{pq}$, $p=2$ or $3$, and $m=1$
  \item $k$ is even, $t=0$, $p \not\equiv 1 \pmod 4$, and $m=1$
  \item $k$ is even, $t=\pm \sqrt{q}$, $p \not\equiv 1 \pmod 3$, and $m=1$
  \item $k$ is even, $t=\pm 2\sqrt{q}$, and $m=n=\sqrt{q}\mp 1$.
\end{enumerate}
\end{lemma}

We note that, the Case~1 in Lemma~\ref{lem:GS} corresponds to ordinary elliptic curves and the other cases corresponds to suppersingular elliptic curves.

As usual, we let $E[l]$ be the set of $l$-torsion points of the elliptic curve $E$ over $\F_q$, that is,
$$
E[l]=\set{P~:~P\in E(\overline{\F}_q),\ lP=\cO}.
$$
We note that, if $\gcd(l,q)=1$, then
$$E[l] \cong \Z/l\Z \times \Z/l\Z.$$

\begin{lemma}
\label{lem:EndG}
Let $E$ be an ordinary elliptic curve over $\F_q$. The following are equivalent:
\begin{enumerate}
  \item $m=\max \set { l~:~ l \in \N, \ \gcd(l,q)=1,\ E[l] \subseteq E(\F_q)}$
  \item $m=\max \set{ l~:~ l \in \N, \ l \mid q-1,\ l^2 \mid \#E(\F_q),\ O\(\frac{\Delta}{l^2}\) \subseteq \End(E) }$
  \item $E(\F_q) \simeq \Z_m \times \Z_n$, where $m \mid n$ and $m \mid q-1$.
\end{enumerate}
\end{lemma}

\begin{proof}
We recall,~\cite[Prposition~3.7]{Sch}, that for all positive integers $l$ with $\gcd(l,q)=1$, we have $E[l] \subseteq E(\F_q)$ if and only if $\ l \mid q-1$, $l^2 \mid \#E(\F_q)$ and $O\(\frac{\Delta}{l^2}\) \subseteq \End(E)$. Therefore, the descriptions of $m$ in Cases~1 and~2 are the same.

Moreover, for all positive integers $l$ with $\gcd(l,q)=1$, we have $E[l] \simeq \Z_l \times \Z_l$. Suppose $E(\F_q) \simeq \Z_m \times \Z_n$, where $m \mid n$ and $m \mid q-1$. Then, for all $l$ with $\gcd(l,q)=1$, we have $E[l] \subseteq E(\F_q)$ if and only if $l \mid m$. Hence,
Cases~1 and~3 are also equivalent.
\end{proof}

We recall the definition  of  the numbers $c_t$ and $s_t$ and of the
sets $\cS_t(m)$ given in Section~\ref{sec:intro}.

\begin{lemma}
\label{lem:EndE}
Let $E$ be an ordinary elliptic curve over $\F_q$. Assume that $m, n$ are positive integers with $m\mid n$, $m \mid q-1$ and $mn=\#E(\F_q)=q+1-t$. Then, we have $E(\F_q) \simeq \Z_m \times \Z_n$
if and only if
$$
\End(E)= O\(\frac{\Delta}{l^2}\)
$$
for some $l\in \cS_t(m)$.
\end{lemma}

\begin{proof}
We note that
$$
\End(E)= O\(\frac{\Delta}{l^2}\),
$$
where $l$ is some positive divisor of $c_t$. By assumption, $m$ is a divisor of $s_t$.
>From Lemma~\ref{lem:EndG}, we have $E(\F_q) \simeq \Z_m \times \Z_n$ if and only if $m$ is the largest divisor of $s_t$  satisfying $O\(\frac{\Delta}{m^2}\) \subseteq \End(E)= O\(\frac{\Delta}{l^2}\)$. The latter is equivalent to have $l \in \cS_t(m)$ which completes the proof.
\end{proof}

\subsection{Primes in arithmetic progressions}

For a real $z \ge 2$ and integers $s> r \ge 0$
we denote by $\pi(z;s,r)$  the number of primes $p \le z$
such that  $p \equiv r \pmod s$.

An asymptotic estimate of the number of primes in arithmetic progressions is given by the {\sl Siegel--Walfisz theorem}, see~\cite[Theorem~1.4.6]{CrPom}.

\begin{lemma}
\label{lem:SigWalf}
For every fixed  $A>0$ there exists $C>0$ such that for $z\ge 2$ and for all positive integers $s\le(\log z)^A$,
$$
\max_{\gcd(r,s)=1}\left|\pi(z;s,r)-\frac{\mathrm {li}\, z}{\varphi(s)}\right| = O\( z \exp\(-C\sqrt{\log z}\)\),
$$
where
$$
\mathrm{li}\,z = \int_{2}^z \frac{d\,u}{\log u}.
$$
\end{lemma}

\section{Proofs}

\subsection{Proof of Theorem~\ref{thm:Gmn}}

We note that $G(q;m,n)\ne 0$ if and only if $m,n$ satisfy one of the cases given by Lemma~\ref{lem:GS}. So, we study the nonzero number $G(q;m,n)$ for the possible values of $m,n$ as follows.

For Case~1, we assume that $\gcd(t,p)=1$ and $t^2\le 4q$. From Lemma~\ref{lem:GS}, we see that $G(q;m,n)\ne 0$ if and only if $m \mid n$ and $m \mid q-1$. So, let $m, n$ be positive integers satisfying the latter conditions. From Lemma~\ref{lem:EndE}, for all elliptic curve $E$ over $\F_q$, we have $E(\F_q) \simeq \Z_m\times \Z_{n}$ if and only if $\End(E)=O(\frac{\Delta}{l^2})$ for some $l\in \cS_t(m)$. We also note that, all orders $O(\frac{\Delta}{l^2})$ whit $l\in \cS_t(m)$, will occur as the endomorphism ring of some elliptic curves over $\F_q$, see~\cite[Theorem 4.2]{Water}. Moreover, the number of $\F_q$ isomorphism classes of elliptic curves with $\End(E)=O(\frac{\Delta}{l^2})$ is $h\(\frac{\Delta}{l^2}\)$ (e.g.~see~\cite[Theorem~4.5]{Sch,Water}). Therefore, we have
$$G(q;m,n)=\sum_{l\in \cS_{t}(m)} h\(\frac{\Delta}{l^2}\).
$$

For Case~2, we have $t=0$. Moreover, $G(q;m,n)$ with $m=1$ is the number of cyclic supersingular elliptic curves over $\F_q$ with trace $0$ (up to $\F_q$-isomorphism), that is, $h(-4p)$, see~\cite[Lemma 4.8]{Sch}.

For Case~3, we have $t=0$ and $q \equiv 3 \pmod 4$.
Also, $G(q;m,n)$ with $m=2$ is the number of non-cyclic supersingular elliptic curves over $\F_q$ with trace $0$ (up to $\F_q$-isomorphism). This is $H(-4p)-h(-4p)=h(-p)$.

For other cases, we have $t^2=q,2q,3q,4q$. Also, all supersingular elliptic curves in the corresponding isogeny class are cyclic. Then, $G(q;m,n)$ with $m=1$ is the isogeny class number given by Lemma~\ref{lem:isoN}. So, the proof of Theorem~\ref{thm:Gmn} is complete.

\subsection{Proof of Theorem~\ref{thm:G}}

The possible group structures of elliptic curves over $\F_q$ are the groups isomorphic to $\Z_m \times \Z_n$, for some values $m,n$ described by Lemma~\ref{lem:GS}. For an integer $t$, let $f(q;t)$ be the the number of distinct group structures of elliptic curves over $\F_q$ with the trace $t$. Let $t$ be a positive integer with $\abs{t} \le 2\sqrt{q}$. From Lemma~\ref{lem:GS}, we consider the following cases for~$t$.
\begin{enumerate}
\item Suppose $\gcd(t,p)=1$. Let $N=q+1-t$. Then, the group $\Z_m \times \Z_n$, for positive integers $m,n$ with $m \le n$, is the group structure of some elliptic curve $E$  over $\F_q$  with trace $t$ if and only if $m \mid n$, $m \mid q-1$ and $mn=N$. This is equivalent to have $m^2 \mid N$, $m \mid q-1$ and $mn=N$. As before, let $s_t$ be the largest integer such that $s_t^2 \mid N$ and $s_t \mid q-1$. Therefore, there is a one to one correspondence between the group structures of ordinary elliptic curves over $\F_q$ with the trace $t$ and positive integer divisors of $s_t$. So,
    \begin{equation}\label{eq:g(q;t)}
    f(q; t)=d(s_t)
    \end{equation}
if $\gcd(t,p)=1$.

\item Suppose $t\mid p$. Then, we may have $t^2/q=0,1,2,3$ or $4$. From Lemma~\ref{lem:GS}, we see that
\begin{equation}\label{eq:g(q;t)0}
f(q;t)=
\left\{
\begin{array}{ll}
\displaystyle{
1+\frac{1-\chi_p(-1)}{2},}&  \text{ if } k \text{ is odd,}\ t=0, \\
1, &  \text{ if } k \text{ is odd, } p = 2 \text{ or } 3,\ t^2=pq.\\
\displaystyle{\frac{1-\chi_p(-1)}{2},} &  \text{ if } k \text{ is even, } p \ne 2,\ t=0, \\
1, &  \text{ if } k \text{ is even, } p = 2,\ t=0,\\
\displaystyle{\frac{1-\chi_p(-3)}{2},} &  \text{ if } k \text{ is even, } p \ne 3,\ t^2=q, \\
1, &  \text{ if } k \text{ is even, } p =3,\ t^2=q, \\
1, &  \text{ if } k \text{ is even, } t^2=4q, \\
0, &  \text{ otherwise.}
\end{array}
\right.
\end{equation}
\end{enumerate}
Now, we sum up $g(q; t)$ over all possible integer values of $t$. We have $$F(q)=\sum_{t\in \Z,\ t^2\le 4q} f(q;t).$$
Using~\eqref{eq:g(q;t)} and~\eqref{eq:g(q;t)0}, we obtain the explicit formulas for $F(q)$.

\subsection{Proof of Theorem~\ref{thm:LUG}}

Let $\cH_q$ be the set of integers of the Hasse-Weil interval, that is,
$$\cH_q=\set{N: N \in \N,\;  q-2\sqrt{q}+1 \le N \le q+2\sqrt{q}+1 }.$$
We recall, from the proof of Theorem~\ref{thm:G}, that for every $N\in \cH_q$ with $\gcd(N-1,p)=1$, there is a bijection between the set of group structures of isogenous elliptic curves $E$ over $\F_q$ with order $N$ and the set of positive divisors $m$ of $q-1$ with $m^2 \mid N$.

For a positive integer divisor $m$ of $q-1$, let $g(q;m)$ be the number of distinct group structures $\Z_m \times \Z_n$ of elliptic curves over $\F_q$ for some $n\in \N$. In other words, $g(q;m)$ is the cardinality of the set of positive integers $n$ where there exists some elliptic curve $E$ over $\F_q$ with $E(\F_q) \simeq \Z_m \times \Z_n$.
Clearly, we have
\begin{equation}
\label{eq:Gg}
F(q)=\sum_{m\mid q-1} g(q;m).
\end{equation}

Here, we express $g(q;m)$ by counting the number of multiples of $m^2$ in $\cH_q$.
For a positive integer divisor $m$ of $q-1$, let
$$\cH_q(m)=\set{N: N \in \cH_q,\ \gcd(N-1,p)=1,\ m^2\mid N }.$$
>From the proof of Theorem~\ref{thm:G} and by Lemma~\ref{lem:GS}, for all positive divisors $m$ of $q-1$, we have
\begin{equation}\label{eq:g(q,m)}
g(q;m)=\#\cH_q(m) + \delta_q(m),
\end{equation}
where
$$
\delta_q(m)=
\left\{
\begin{array}{ll}
1, &  \text{ if } k \text{ is odd, } p \ne 2, 3,\ m=1, \\
\displaystyle{1+\frac{1-\chi_p(-1)}{2}-\chi_p(-3),} &  \text{ if } k \text{ is even, } p \ne 2, 3,\ m=1, \\
3, &  \text{ if } p = 2 \text{ or } 3,\ m=1, \\
\displaystyle{\frac{1-\chi_p(-1)}{2}, }&
 \text{ if } k \text{ is odd,}\ m=2 \\
1, &  \text{ if } k \text{ is even, } m=\sqrt{q} \pm 1, \\
0, & \text{ otherwise. }
\end{array}
\right.
$$
Next, using~\eqref{eq:Gg} and~\eqref{eq:g(q,m)}, we obtain
\begin{equation}
\label{eq:GHdel}
F(q)=\sum_{m \mid q-1} \#\cH_q(m)+ \delta_q(m).
\end{equation}
We note that $\#\cH_q=2\[2\sqrt{q}\]+1$. Moreover, for all divisors $m$ of $q-1$, if $m \ge \sqrt{q}+1$ then $\#\cH_q(m)=0$ and if $m < \sqrt{q}+1$, then
$$\[\frac{4\sqrt{q}}{m^2}\]-\[\frac{4\sqrt{q}}{m^2p}\]-1 \le \#\cH_q(m) \le \[\frac{4\sqrt{q}}{m^2}\]-\[\frac{4\sqrt{q}}{m^2p}\]+1$$
and so,
\begin{equation}
\label{eq:Hq(m)}
\frac{4\sqrt{q}}{m^2}\(1-\frac{1}{p}\)-2 < \#\cH_q(m) < \frac{4\sqrt{q}}{m^2}\(1-\frac{1}{p}\)+2.
\end{equation}

To obtain an upper bound for $F(q)$, we write
\begin{eqnarray*}
\lefteqn{
\sum_{m \mid q-1} \#\cH_q(m)}\\
& &\qquad = \sum_{\substack{m \mid q-1,\\m < \sqrt{q-1}}} \#\cH_q(m)+  \sum_{\substack{m \mid q-1,\\ \sqrt{q-1} \le m < \sqrt{q}+1}} \#\cH_q(m)+  \sum_{\substack{m \mid q-1,\\m \ge \sqrt{q}+1}} \#\cH_q(m) \\
& &\qquad  < \sum_{m \mid q-1} \frac{4\sqrt{q}}{m^2}\(1-\frac{1}{p}\)+ \sum_{\substack{m \mid q-1,\\m < \sqrt{q-1}}} 2+ \sum_{\substack{m \mid q-1,\\ \sqrt{q-1} \le m < \sqrt{q}+1}} \#\cH_q(m).
\end{eqnarray*}

One can see that,
$$
\sum_{\substack{m \mid q-1,\\ \sqrt{q-1} \le m < \sqrt{q}+1}} \#\cH_q(m)+ \sum_{m \mid q-1} \delta_q(m) \le 5.
$$
Then, from~\eqref{eq:GHdel}, we have
$$
F(q) < 4\sqrt{q} \(1-\frac{1}{p}\) \sum_{m \in \N} \frac{1}{m^2} + d(q-1)+5 = \frac{2\pi^2}{3}\sqrt{q}\(1-\frac{1}{p}\)+d(q-1)+5.
$$

Now, we provide the lower bound for $F(q)$.
If $p=2$, using~\eqref{eq:GHdel}, we write
$$F(q) \ge  \#\cH_q(1)+ \sum_{m \mid q-1} \delta_q(m) \ge \[2\sqrt{q}\] + 3 > 2\sqrt{q}+2 .$$
For $p \ge 3$, using~\eqref{eq:GHdel}, we write
$$F(q) \ge  \#\cH_q(1)+\#\cH_q(2)+ \sum_{m \mid q-1} \delta_q(m).$$
We see from~\eqref{eq:Hq(m)} that
$$
\#\cH_q(m) > \frac{4\sqrt{q}}{m^2} \(1-\frac{1}{p}\)-2.
$$
Moreover, one can see that
$$\#\cH_q(2) > \sqrt{q} \(1-\frac{1}{p}\)-1
$$
if $q \equiv 1 \pmod 4$.
From~\eqref{eq:g(q,m)}, we have
$$\sum_{m \mid q-1} \delta_q(m) \ge 1.
$$
Furthermore,
$$\sum_{m \mid q-1} \delta_q(m) \ge 2
$$
if $q \equiv 3 \pmod 4$.
Therefore, $$\#\cH_q(2)+\sum_{m \mid q-1} \delta_q(m) > \sqrt{q} \(1-\frac{1}{p}\),$$
which completes the proof.

\subsection{Proof of Theorem~\ref{thm:limG}}

To prove the result of Case~1,
let us
choose a sufficiently large integer $L$, and
let $M$ be the least common multiple of all positive
integers $m \le L$.

We now choose a prime $p \equiv 1 \pmod M$
and put $q = p$.
Using~\eqref{eq:Gg} and~\eqref{eq:g(q,m)} we derive
$$
F(q) \ge \sum_{\substack{m\mid q-1\\ m \le L}} g(q;m)
= \sum_{m \le L} g(q;m)
=  \sum_{m \le L} \(\#\cH_q(m) + O(1)\)
$$
Since by~\eqref{eq:Hq(m)} for $q=p$ we have
$$
\#\cH_q(m) = \frac{4 \sqrt{q}}{m^2} + O(1),
$$
we now derive
\begin{eqnarray*}
F(q) &\ge  &
  \sum_{m \le L} \( \frac{4 \sqrt{q}}{m^2} + O(1)\)
= 4 \sqrt{q}\sum_{m \le L} \frac{1}{m^2} + O(L)\\
& =  & 4 \sqrt{q}\(\frac{\pi^2}{6} + O(1/L)\)+ O(L).
\end{eqnarray*}
Since by the prime number theorem we have
$$
q \ge M \ge \exp\((1 + o(1))L\),
$$
taking $L \to \infty$ we obtain
$$
F(q) = \(\frac{2\pi^2}{3}+o(1)\)\sqrt{q}=
\(\frac{2\pi^2}{3}+o(1)\)\sqrt{q}\(1 - \frac{1}{p}\)
$$
for the above sequence of $q=p$.

For Case~2, we recall a result of Heath-Brown~\cite{HB}, which
asserts that there are infinitely many primes $p$ such that
either $p=2 \ell +1$ for a prime $\ell$
or  $p=2 \ell_1 \ell_2+ 1$  for a primes $\ell_1, \ell_2
\ge p^{\alpha}$ for some
$\alpha > 1/4$ (one can take $\alpha = 0.276\ldots$, see the proof
of~\cite[Lemma~1]{HB}). Using~\eqref{eq:Gg} and~\eqref{eq:g(q,m)}
we see that for each such prime $p$ and $q=p$ we have
\begin{eqnarray*}
F(q) &= & \sum_{m\mid q-1} g(q;m)
= \sum_{m\mid q-1} \(\#\cH_q(m) + O(1)\)\\
 &= &\#\cH_q(1)+ \#\cH_q(2) + O(1) = 5\sqrt{q} + O(1) \\
 &= &5\sqrt{q}\(1 - \frac{1}{p}\) + O(1).
\end{eqnarray*}

Finally in Case~3, we recall that if $q = 2^r$, where $r$ is prime then
all prime divisors $\ell$ of $q-1$ satisfy
$\ell \equiv 1 \pmod r$ (since $r$ is the multiplicative
order of $2$ modulo $\ell$, thus $r \mid \ell-1$).
In particular for any $m \mid q-1$ with $m > 1$ we have $m > r$.
Hence, as before,  and also recalling~\eqref{eq:Hq(m)},
for $q = 2^r$ we obtain
\begin{eqnarray*}
F(q) &= & \
 \#\cH_q(1) +
\sum_{\substack{m\mid q-1 \\ m > 1}} \(\#\cH_q(m) + O(1)\)\\
 &= &\#\cH_q(1)+  O\(\sum_{\substack{m\mid q-1 \\ m > 1}}
 \(q^{1/2}m^{-2} + 1\)\) \\
&= &\#\cH_q(1)+  O\(q^{1/2}\sum_{ m > r} m^{-2} + d(q-1)\) \\
&= & \#\cH_q(1)+  O\(q^{1/2} (\log q)^{-1} + d(q-1)\) =
\(2+o(1)\) q^{1/2}
\end{eqnarray*}
which concludes the proof.

\subsection{Proof of Theorem~\ref{thm:asympG}}

Since there are $O(Q^{1/2})$ prime
powers $q =p^k \le Q$ with $k\ge 2$, using the upper
bound of Theorem~\ref{thm:LUG} we obtain
\begin{equation}
\label{eq:q and p}
\sum_{q \le Q} F(q)=  \sum_{p \le Q} F(p) + O(Q).
\end{equation}

We see from~\eqref{eq:GHdel}  that
$$
F(p)=4\sqrt{p}  \sum_{m \mid p-1}\frac{1}{m^2} +
O(d(p-1)).
$$
We recall  the well-known estimate on the divisor function
\begin{equation}
\label{eq:div}
d(s) = s^{o(1)}, \qquad  s\to \infty,
\end{equation}
see~\cite[Theorem~317]{HardyWright}.
Thus
\begin{eqnarray*}
\sum_{p \le Q} F(p)& = &
4\sum_{p \le Q} \sqrt{p}  \sum_{m \mid p-1}\frac{1}{m^2}
+O\(Q^{1 + o(1)}\)\\
& = & 4\sum_{m \le Q}\frac{1}{m^2}
\sum_{\substack{p \le Q \\ p\equiv 1 \pmod m}} \sqrt{p}
+O\(Q^{1 + o(1)}\).
\end{eqnarray*}

By Lemma~\ref{lem:SigWalf}, and partial summation, we see that for
$m \le \log Q$ we have
$$
\sum_{\substack{p \le Q \\ p\equiv 1 \pmod m}} \sqrt{p}
= \(\frac{2}{3}+o(1)\) Q^{1/2} \frac{\mathrm {li}\, Q}{\varphi(m)}
= \(\frac{2}{3}+o(1)\)\frac{Q^{3/2}}{\varphi(m) \log Q}.
$$
Furthermore, for $ m > \log Q$ we use the trivial estimate
$$
\sum_{\substack{p \le Q \\ p\equiv 1 \pmod m}} \sqrt{p}
\le Q^{1/2} \sum_{\substack{2 \le n \le Q \\ n\equiv 1 \pmod m}} 1
= O( Q^{3/2}m^{-1}).
$$
Therefore
\begin{eqnarray*}
\sum_{p \le Q} F(p)&=& \(\frac{8}{3}+o(1)\)\frac{Q^{3/2}}{\log Q}
\sum_{m \le \log Q} \frac{1}{m^2 \varphi(m)}
+ O\(Q^{3/2}\sum_{m> \log Q} m^{-3}\)\\
&=&\(\frac{8}{3}+o(1)\)\frac{Q^{3/2}}{\log Q}
\sum_{m =1}^\infty \frac{1}{m^2 \varphi(m)}
+ O\(Q^{3/2} (\log Q)^{-2}\),
\end{eqnarray*}
which together with~\eqref{eq:q and p}
concludes the proof.

\section{Distribution of the Most Frequent Group Structures}
\label{sec:max freq}

\subsection{Preliminaries}

Here we present some numerical data concerning
the values of $G(q)$ given by~\eqref{eq:Hq}
and also about the values of $m$ and $n$ at which these
values are achieved. Furthermore, we concentrate here on
prime values $q=p$.

First of all we note that for any $N$ we have
$$
\sum_{\substack{m,n \ge 1\\mn = N}} G(p;m,n)
= I(p;N),
$$
where as before $I(p;N)$ is  the number of distinct isomorphism
classes of elliptic curves $E$ over $\F_p$  (up to isomorphism over
$\F_p$)  such that $\# E(\F_p) = N$ (see Lemma~\ref{lem:isoN}). In
particular
\begin{equation}
\label{eq:H and I}
\max_N I(p;N)/d(N) \le G(p) \le \max_N I(p;N).
\end{equation}

It is well-known that the bounds on the Kronecker class number
imply that
$$
I(p;N) \ll p^{1/2} \log p ( \log \log p)^2
$$
and for all $N \in [p +1 -p^{1/2}, p +1  + p^{1/2}]$
but maybe at most two of them, we have
$$
I(p;N) \gg p^{1/2}/ \log p ,
$$
see, for example,~\cite[Proposition~1.9]{Len}. Thus,
recalling~\eqref{eq:div}, we derive  from the
inequalities~\eqref{eq:H and I} that
\begin{equation}
\label{eq:Gp:rough}
G(p) = p^{1/2 +o(1)}.
\end{equation}

\subsection{Numerical data}

We see that from~\eqref{eq:Gp:rough} that it is natural to study the
 values of $G(p)$ scaled by $p^{1/2}$.
In fact, our experiments
with $41538$ primes $p< \pmax$ show that scaling by $p^{1/2}\log p$
is more natural and  the ratio $G(p)/p^{1/2}\log p$ stabilises
in a reasonably narrow strip between roughly $0.1$ and $0.2$,
see Figure~\ref{fig:maxG}.

\begin{figure}[htb]
\begin{center}
\includegraphics*[width=\textwidth]{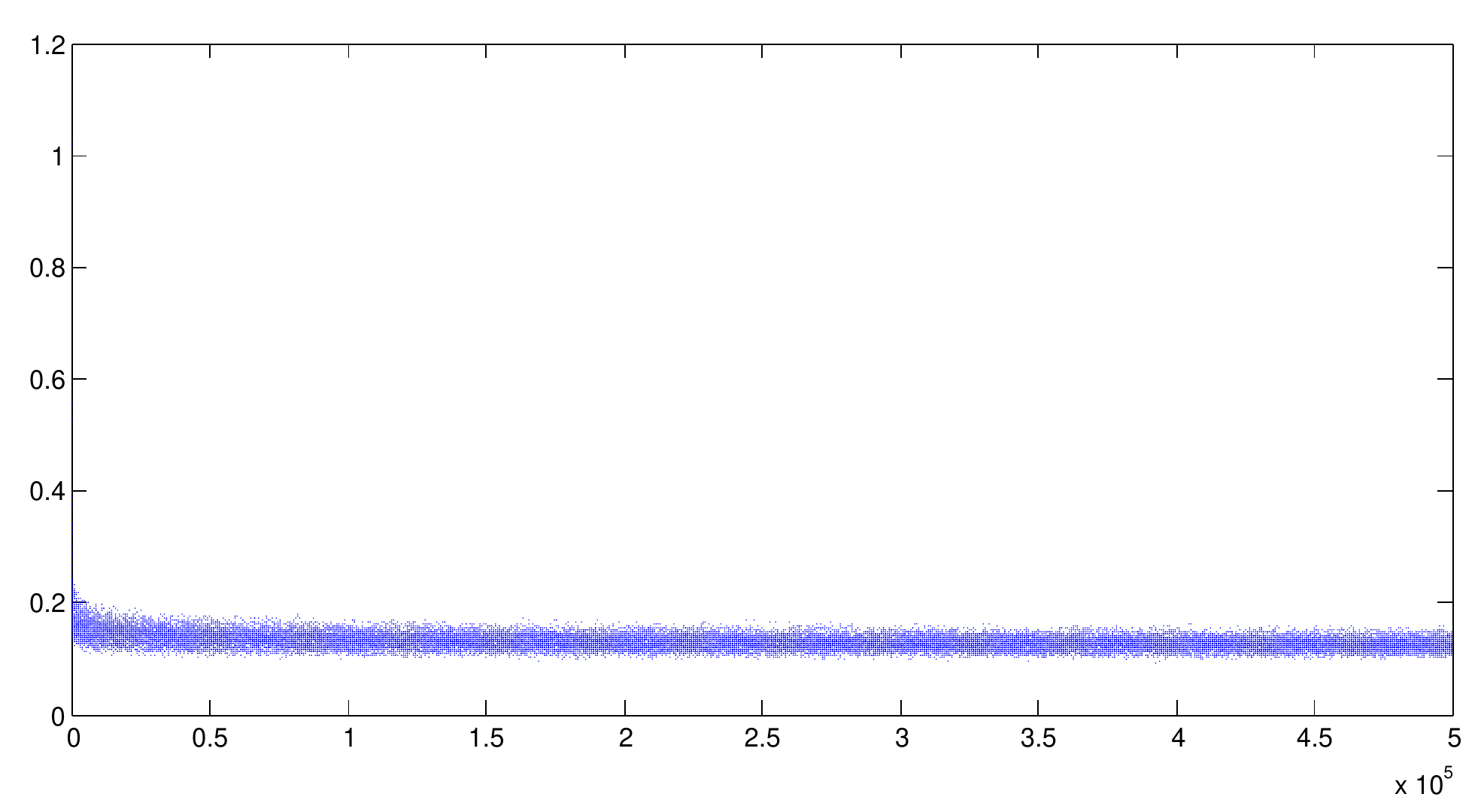}
\end{center}
\vspace*{-10mm} \caption{Distribution of $G(p)/p^{1/2}\log p$ for
primes $p<\pmax$.} \label{fig:maxG}
\end{figure}

We also notice that for all primes checked the value of $G(p)$ is
always achieved for $(m,n)$ with $m = 1$ (that is, for curves with
cyclic group of points). Moreover, for some primes the same value is
also achieved for some pairs $(m,n)$ with $m=2$. In our experiments
the value of $G(p)$ has never been achieved with $m \ge 3$.

We also compare $G(p)$ with
$$
I(p) = \max_{N} I(p;N),
$$
see Figure~\ref{fig:GI}.

\begin{figure}[htb]
\begin{center}
\includegraphics*[width=\textwidth]{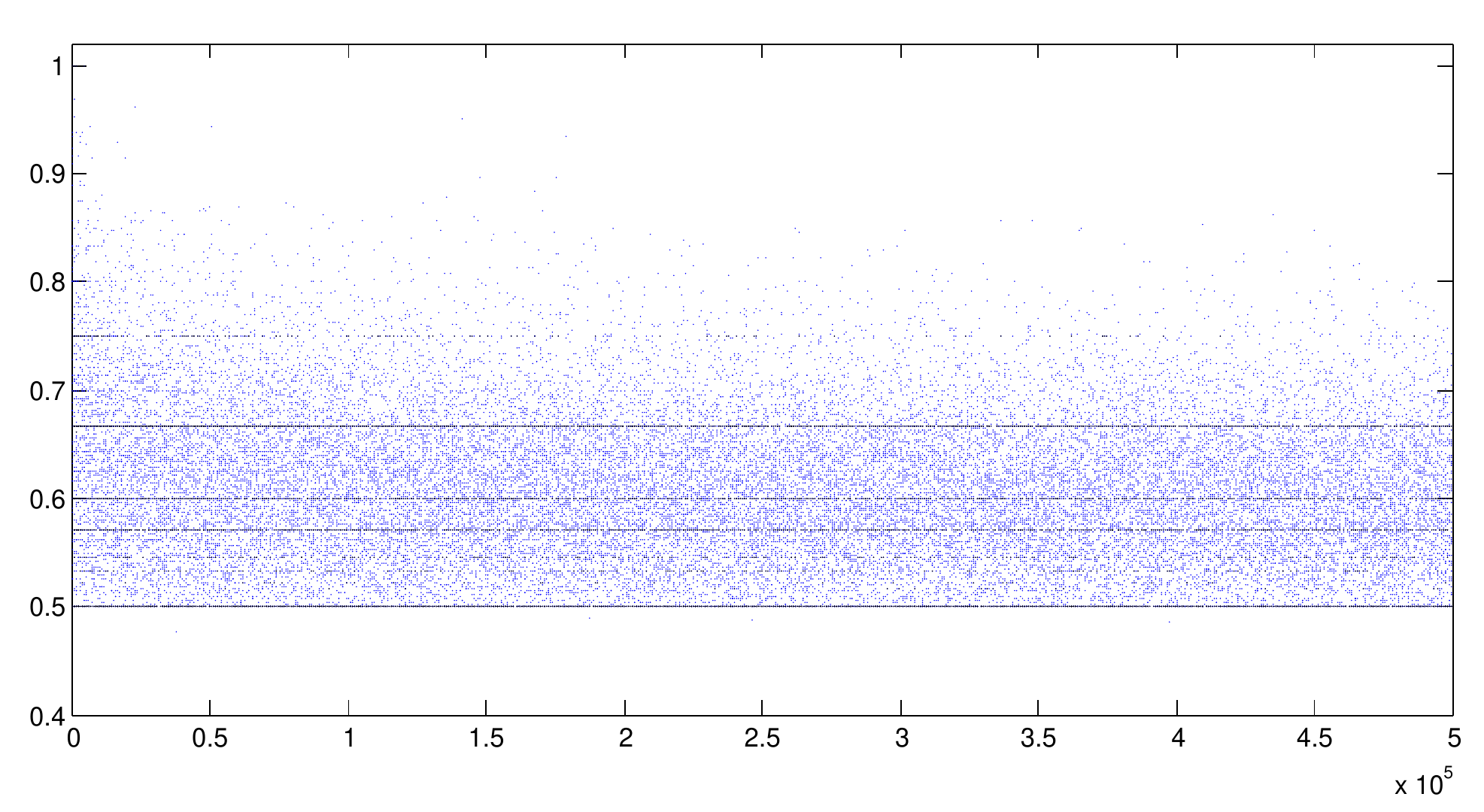}
\end{center}
\vspace*{-10mm} \caption{Distribution of $G(p)/I(p)$ for
primes $p<\pmax$.} \label{fig:GI}
\end{figure}

For primes $p < \pmax$, where we have computed $G(p)/I(p)$ this
ration has achieved $1$ for  $p 2, 5, 7, 17, 29, 41, 101, 1009,
1109, 1879, 4289$, where $G(p)$ and $I(p)$ are achieved with the same
value of $t$ with $s_t=1$. Also, only four times (for $p =37591,
187651, 246391, 397591$) the value of $G(p)/I(p)$ has been
below $0.5$. Unfortunately these extreme values of both types are
invisible on Figure~\ref{fig:GI}. We do not know whether these
primes are just some sporadic exceptions or whether there are
infinitely many such primes. More generally, it is certainly
interesting to evaluate or at least obtain nontrivial theoretic
estimates on
$$
\limsup_{p\to \infty}G(p)/I(p) \mand
\liminf_{p\to \infty}G(p)/I(p).
$$
This may also help to explain the presence of several horizontal
lines on  Figure~\ref{fig:GI} (slightly emphasised there to improve
their visibility).

Clearly, one expects that the value of $G(p)$ is achieved for
$(m,n)$ for which $t= p+1 -N$, where $N = mn$, is small, so that
$\Delta=t^2 - 4p$ has a large absolute value which leads to a large
value of $I(p;N)$.
 However this is offset by the fact that for $N$
having many divisors so the  value of $I(p;N)$ is ``split'' between
$d(s_t)$  values of $G(p;m,n)$. This effect is observed in the
numerical results which are presented below, which show that if
$G(p;m,n) = G(p)$ then $t= p+1 -mn$ is small but not necessary very
small. In particular, most of the time $G(p)$ and $I(p)$ are
achieved on different values of $t$, namely in about $82.2\%$ of the
cases within the above range of primes $p$
(more precisely for  $34158$ primes out of the total
number of $41538$ primes $p< \pmax$). Furthermore, it seems
that the remaining $7380$ cases
in which $G(p)$ and $I(p)$ are achieved on the same
value of $t$, are the ones that mainly (but not entirely)
responsible for the presence of horizontal lines on
Figure~\ref{fig:GI}. Indeed, the same lines are clearly visible on
Figure~\ref{fig:GI-t} where the ratios $G(p)/I(p)$  are plotted only
if they come from the same value of $t$.

\begin{figure}[htb]
\begin{center}
\includegraphics*[width=\textwidth]{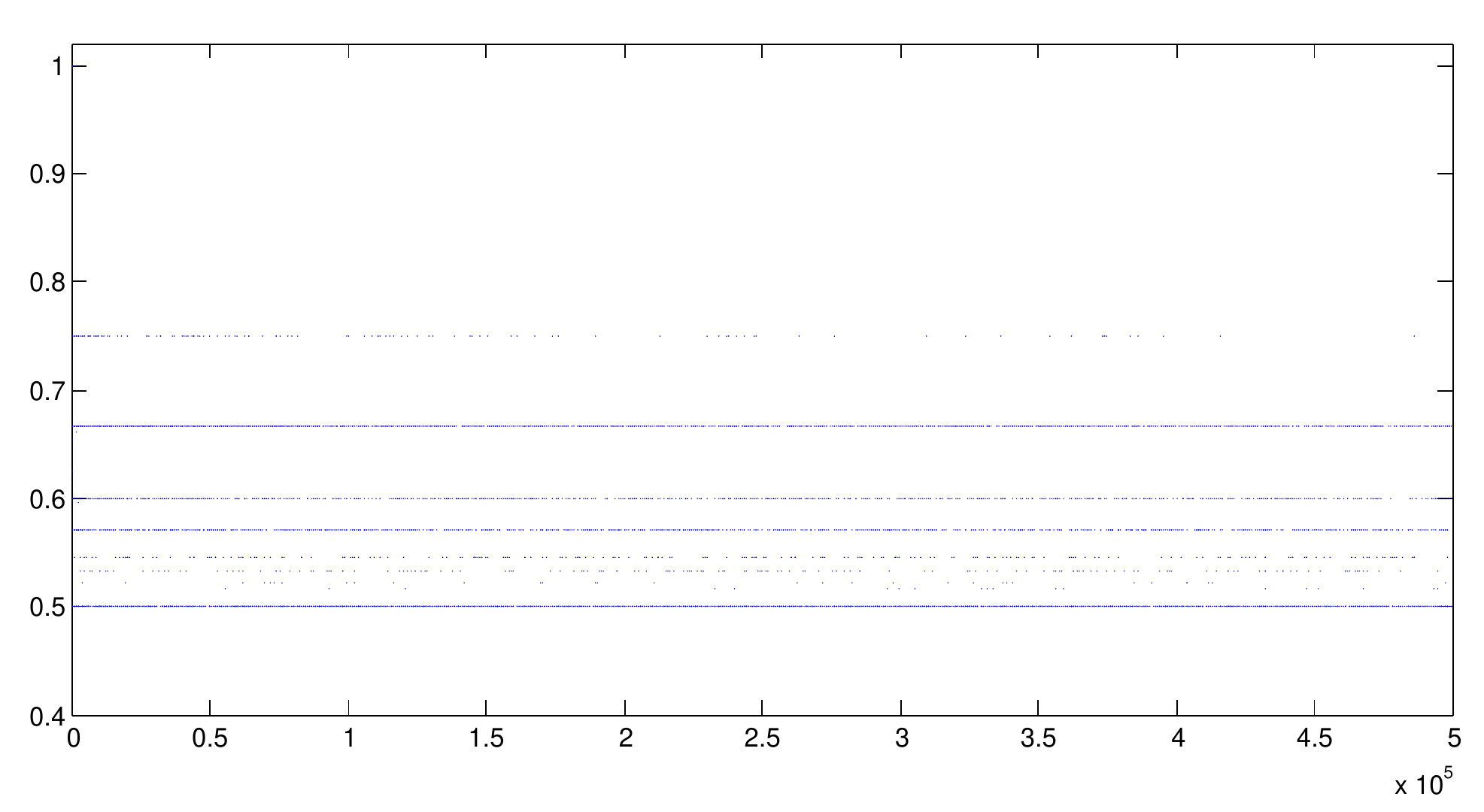}
\end{center}
\vspace*{-10mm} \caption{Distribution of $G(p)/I(p)$ achieved on the
same value of $t$ for primes $p<\pmax$.} \label{fig:GI-t}
\end{figure}

We also summarise this in Table~\ref{tab:GI-t} which gives the
number of points on horizontal lines on Figure~\ref{fig:GI} and
Figure~\ref{fig:GI-t} (ordered by the the total number of points).

\begin{table*}[!ht]
\begin{center}
\begin{tabular}{|l|l|l|}
\hline
Ration $G(p)/I(p)$    &  Number of primes & Number of primes \\
&  on Figure~\ref{fig:GI} &  on Figure~\ref{fig:GI-t}\\
\hline
 1/2   & 2933  &    2931\\
 2/3   & 2300 &   1968 \\
  3/5    & 1287 &     883\\
 4/7   & 1236 &    1012 \\
 6/11  & 329 &     220 \\
  5/8   & 292 &  1 \\
 8/15  & 268 &     161 \\
 3/4   & 258 &     139 \\
 12/23 & 71 &     32  \\
  17/26 & 45 &   1 \\
 16/31   & 39 &    19\\
  45/68  & 28 & 1 \\
  28/47  & 15 & 1 \\
 1    &  11 & 11 \\
\hline
\end{tabular}
\end{center}
\vspace*{-5mm} \caption{Rations of $G(p)/I(p)$ for primes $p<\pmax$.}\label{tab:GI-t}
\end{table*}

Let 
$$\cT_{\max}(p)=\set{t~:~t=p+1-mn,\ G(p;m,n)=G(p)},
$$
that is, for each $p$, $\cT_{\max}(p)$ is the set of the traces
corresponding to the most ``popular'' group structures. 
In Table~\ref{tab:Tmax} we give some data about the distribution of 
$\# \cT_{\max}(p)$ for primes $p < 500,000$. 
In particular $\# \cT_{\max}(p)=1$ in about $52\%$
of the cases.

\begin{table*}[h!]
\begin{center}
\begin{tabular}{|l|l|}
\hline
$\# \cT_{\max}(p)$   &  Number of primes\\
\hline
1  & 21638\\
2  & 19087\\
3  & 230\\
4  & 524\\
5 & 19\\
6  & 36\\
7  & 3\\
10  & 1\\\hline
\end{tabular}
\end{center}
\vspace*{-5mm} \caption{Distribution of 
$\# \cT_{\max}(p)$  for primes $p<\pmax$.}\label{tab:Tmax}
\end{table*}

We also remark that the set $\cT_{\max}(p)$  is 
symmetric around $0$ (that is, $\cT_{\max}(p)=-\cT_{\max}(p)$)
for 20020 primes out of the total number of 
$41538$ primes $p< \pmax$.

Figure~\ref{fig:Tmax} presents the scaled values $t/p^{1/2}$, where
$t\in T_{\max}$, for primes $p < \pmax$.

\begin{figure}[h!]
\begin{center}
\includegraphics*[width=\textwidth]{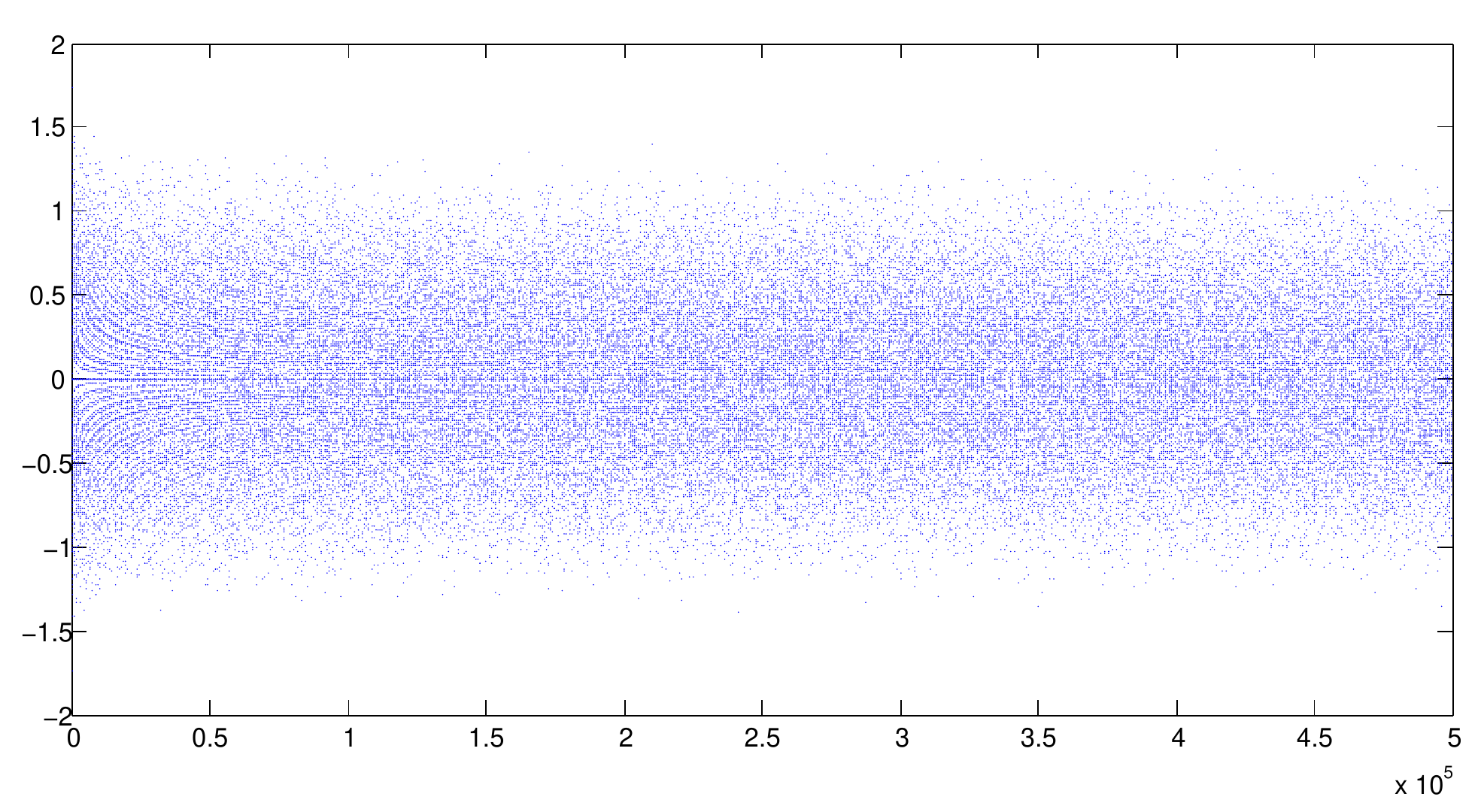}
\end{center}
\vspace*{-10mm} \caption{Distribution of $\frac{t}{\sqrt{p}}$, where
$t\in T_{\max}$, for primes $p<\pmax.$} \label{fig:Tmax}
\end{figure}

As we have mentioned, we do not have any solid theoretic explanation
to the observed facts. There is certainly more to investigate here,
numerically and theoretically,  in order to understand the behaviour
of  $G(p;m,n)$.

\end {document}